\documentclass[12pt]{amsart}
\usepackage{latexsym, url}
\usepackage{amsthm}
\usepackage{amsmath}
\usepackage{amsfonts}
\usepackage{amssymb}
\usepackage[final]{showkeys}
\usepackage{xypic}
\input{xy}
\xyoption{all}

\newtheorem{theo}{Theorem}[section]
\newtheorem{lemma}[theo]{Lemma}

\newtheorem{propo}[theo]{Proposition}
\newtheorem{defi}[theo]{Definition}
\newtheorem{coro}[theo]{Corollary}
\newtheorem{rem}[theo]{Remark}

\newtheorem{exam}[theo]{Example}

\newcommand\Ind{\operatorname{Ind}}

\newcommand\Mod{\operatorname{Mod}}

\newcommand\Hom{\operatorname{Hom}}

\newcommand\op{\operatorname{op}}

\newcommand\coex{\operatorname{coex}}
\newcommand\ex{\operatorname{ex}}
\newcommand\Set{\operatorname{\bf Set}}

\newcommand\Ab{\operatorname{\bf Ab}}

\newcommand\ca{\mathcal {A}}
\newcommand\cb{\mathcal {B}}

\newcommand\cu{\mathcal {U}}

\newcommand\ci{\mathcal {I}}

\newcommand\ck{\mathcal {K}}
\newcommand\cl{\mathcal {L}}
\newcommand\cm{\mathcal {M}}
\newcommand\cn{\mathcal {N}}

\newcommand\cx{\mathcal {X}}

\newcommand\ratls{\mathbb {Q}}
\newcommand\ints{\mathbb {Z}}

\date{January 11, 2020}
 
\begin{document}
\title[Cofibrant generation of pure monomorphisms]
{Cofibrant generation of pure monomorphisms}

\author[Lieberman]{Michael Lieberman}
\email{lieberman@math.muni.cz}
\urladdr{http://www.math.muni.cz/\textasciitilde lieberman/}
\address{Institute of Mathematics, Faculty of Mechanical Engineering, Brno University of Technology, Brno, Czech Republic}
\address{Department of Mathematics and Statistics, Faculty of Science, Masaryk University, Brno, Czech Republic}

\author[Positselski]{Leonid Positselski}
\email{positselski@math.cas.cz}
\urladdr{http://users.math.cas.cz/\textasciitilde positselski/}
\address{Institute of Mathematics of the Czech Academy of Sciences,
\v Zitn\'a~25, 115~67 Prague~1, Czech Republic}
\address{Laboratory of Algebra and Number Theory, Institute for Information
Transmission Problems, Moscow 127051, Russia}
\thanks{The second author is supported by research plan RVO:~67985840}
 
\author[Rosick\'y]{Ji\v r\'i Rosick\'y}
\email{rosicky@math.muni.cz}
\urladdr{http://www.math.muni.cz/\textasciitilde rosicky/}
\address{Department of Mathematics and Statistics, Faculty of Science, Masaryk University, Brno, Czech Republic}
\thanks{The third author is supported by the Grant agency of the Czech Republic under the grant 19-00902S}

\author[Vasey]{Sebastien Vasey}
\email{sebv@math.harvard.edu}
\urladdr{http://math.harvard.edu/\textasciitilde sebv/}
\address{Department of Mathematics \\ Harvard University \\ Cambridge, Massachusetts, USA}
 
\begin{abstract}
We show that pure monomorphisms are cofibrantly generated---generated from a set of morphisms by pushouts, transfinite composition, and retracts---in any locally finitely presentable additive category. In particular, this is true in any category of $R$-modules.
 On the other hand, the classes of all monomorphisms and regular monomorphisms
in a locally finitely presentable additive category need not be well-behaved.
\end{abstract} 

\maketitle

\section{Introduction}
Purity plays a significant role in module theory (see e.g. \cite{P}) and, in particular, in abelian group theory, where it first arose, \cite{Pr}. Pure monomorphisms can be defined in various ways---most typically, in algebraic contexts, in terms of the preservation of exact sequences---but can be characterized as directed colimits of split monomorphisms. Pure injective modules, namely those modules that are injective with respect to pure monomorphisms, form an important class of modules lying between cotorsion and injective modules. Recall that we may associate with any left (right) $R$-module $M$ its character module, $M^+=\Hom_{\ints}(M,\ratls/\ints)$, which is, in turn, a right (left) $R$-module. Any character module is pure injective and the embedding of a module $M$ into its double character module $M^{++}$ makes $M$ a pure submodule of a pure injective module \cite[Proposition 5.3.9]{EJ}. Thus any category of modules has enough pure injectives. \cite{He} proved that any locally finitely presentable additive category has enough pure injectives. Pure injectives also play a role in (non-additive) universal algebra, where they are called equationally compact algebras: in that case, however, one will not necessarily have enough pure injectives.

In the context of module theory, the cotorsion modules and the injective modules form accessible categories.  This follows from the fact that, respectively, the class of flat monomorphisms and the class of all monomorphisms are cofibrantly generated; that is, they can be generated by (retracts of) transfinite compositions of pushouts from a \emph{set} of morphisms.  The next natural question is whether pure monomorphisms, too, are cofibrantly generated: the goal of our paper is to answer this question in the affirmative. As a consequence, we will see that pure injectives form an accessible category. Moreover, we prove the result not merely for categories of modules, but rather for arbitrary locally finitely presentable additive categories. The module case, incidentally, was recently proven in \cite{ST}. Our result also generalizes (via stable independence \cite{LRV}) the model-theoretic stability of the category of $R$-modules with pure embeddings, established in \cite{kucera-mazari}.

We note that locally finitely presentable abelian categories are Grothendieck and recall that the usual proof that Grothendieck categories have enough injectives uses the fact that they have effective unions of subobjects (see \cite{Ba}). Effective unions of pure subobjects were introduced in \cite{BR}, and it was shown there that they imply cofibrant generation of pure monomorphisms. Unfortunately we cannot take this approach, as even the category of abelian groups fails to have effective unions of pure subobjects (see Example~\ref{unions}).  Instead, we employ the (upper)
representation category construction (see \cite{C} or \cite{K}) to embed the starting locally finitely presentable category $\ck$ into a locally finitely presentable coexact category $\cu(\ck)$ in such a way that pure monomorphisms in $\ck$ are sent to regular monomorphisms in $\cu(\ck)$, and the pure injectives of $\ck$ map to precisely the regular injectives of $\cu(\ck)$. Under the assumption that $\ck$ is additive, we will see that $\cu(\ck)$ has enough regular injectives, which gives us the corresponding result for pure injectives in $\ck$.

%% \footnote{We note that such arguments by conversion to a more manageable class of monomorphisms can be found elsewhere in the literature, e.g. \cite{He}. JR: this is not this case.}

 On the other hand, a locally finitely presentable additive category $\ck$ need not have effective unions of subobjects.
Besides, (regular) monomorphisms in $\ck$ need not be stable under
pushouts.
 Consequently, $\ck$ may fail to have enough injectives or regular
injectives.
 We exhibit and discuss some examples of locally finitely presentable
additive categories with bad exactness properties and strange behavior of
injective objects.

\section{Preliminaries}
Locally finitely presentable categories were introduced in \cite{GU}.  The notion of \emph{presentability} was based on the observation that, say, an $R$-module $M$ 
is finitely presented if and only if the associated hom-functor $\Hom(M,-):\Mod_R\to\Set$ preserves directed colimits: in a general category 
$\ck$, objects satisfying the latter condition are called \textit{finitely presentable}. We note that by \emph{directed colimits}, we mean those colimits indexed by a directed poset---\emph{direct limits}, in the older terminology. A cocomplete category $\ck$ is \textit{locally finitely presentable} if it has, up to isomorphism, a set $\ca$ of finitely presentable objects such that every object of $\ck$ is a directed colimit of objects from $\ca$. Every category of 
$R$-modules and every variety of finitary algebras is locally finitely presentable. By replacing $\aleph_0$ by an arbitrary regular cardinal $\kappa$ and directed colimits by $\kappa$-directed ones, we arrive at \textit{locally $\kappa$-presentable categories}. More generally, we say that a category is \textit{locally presentable} if it is locally $\kappa$-presentable for some regular cardinal $\kappa$. By weakening
cocompleteness of $\ck$ to the existence of directed colimits, we arrive at the notion of a  \textit{finitely accessible category} and, analogously,
\textit{$\kappa$-accessible} and \textit{accessible} ones (see \cite{MP} or \cite{AR}).

Purity is usually defined model-theoretically; that is, pure embeddings are those that are elementary with respect to positive-primitive formulas. In locally
finitely presentable categories, there is an alternative, category-theoretic definition which was introduced in \cite{fakir} and simplified to its current form in \cite{AR}.

\begin{defi}\label{defpure}
{
\em
	Let $\ck$ be a category.  We say that a morphism $h:K\to L$ in $\ck$ is \emph{$\aleph_0$-pure} (or, more simply, \emph{pure}) if for any commutative square 
	$$  
      \xymatrix@=3pc{
        K \ar@{}\ar[r]^h{} & L \\
        M \ar [u]^{u} \ar [r]_{f} &
        N \ar[u]_{v}
      }
      $$
      with $M$ and $N$ finitely presentable, $u$ factors through $f$; that is, there is a morphism $g:N\to K$ with $u=gf$.
}
\end{defi}

\begin{rem}\label{purermk}
{
\em
(1) There is no mention of $h$ being a monomorphism in Definition~\ref{defpure}: in any locally finitely presentable category, the pure morphisms will automatically be monomorphisms \cite[2.29]{AR}.  Indeed, they will even be \emph{regular} monomorphisms (\cite[2.31]{AR}. 

(2) In any locally finitely presentable category, pure monomorphisms are precisely directed colimits of split monomorphisms (\cite[2.30]{AR}.

(3) In categories of modules, purity in the sense of Definition~\ref{defpure} corresponds to the usual algebraic formulations of the notion.
}
\end{rem}

We recall a few basic facts about the class of pure monomorphisms (see \cite[p.~85]{AR}):

\begin{rem}\label{pureprops}
{
\em 
In any locally finitely presentable category $\ck$, the class of pure monomorphisms: 
	\begin{enumerate}
		\item Contains all split monomorphisms, hence contains all isomorphisms, and is closed under composition, \cite[2.28]{AR}.
		\item Is closed under directed colimits. In particular, it is closed under \textit{transfinite compositions}.  That is, given
		a smooth chain $(f_{ij}\colon A_i \to A_j)_{i\leq j\leq \lambda}$ (i.e. a chain in which $(f_{ij}\colon A_i \to A_j)_{i<j}$ is a colimit for any limit ordinal $j\leq\lambda$) such that $f_{i,i+1}$ is pure for each $i< \lambda$, then $f_{0\lambda}$ is pure.
		\item Is  \textit{left-cancellable}: if $gf$ is pure, so is $f$, \cite[2.28]{AR}.
		\item Is closed under retracts in the category $\ck^\to$ of morphisms of $\ck$, following (1) and (3) above.
		\item Is stable under pushouts.  By this we mean that in a pushout diagram
$$
\xymatrix@C=3pc@R=3pc{
A \ar[r]^f  \ar[d]_g & B \ar[d]^{\bar g} \\
C\ar [r]_{\bar f} & D
}
$$
$\bar f$ is a pure monomorphism provided that $f$ is a pure monomorphism (see \cite{AR1}).
	\end{enumerate}}
\end{rem}

A class of morphisms in a category $\ck$ is \textit{cofibrantly closed} if it closed under transfinite compositions, pushouts
and retracts in the category $\ck^\to$ (see, e.g., \cite{AHRT}). Hence the class of pure monomorphisms is cofibrantly closed in any locally finitely presentable category. We note that the notion of cofibrant closedness stems from homotopy theory, where the crucial question is whether a given cofibrantly closed class $\cm$ of morphisms is \textit{cofibrantly generated}; that is, if there is a set $\cx$ such that $\cm$ is the closure of $\cx$ under transfinite compositions, pushouts and retracts in the category of morphisms (see \cite{B}). This property is essential for many reasons, not least because cofibrant generation of $\cm$ ensures that it forms the left half of a \emph{weak factorization system} $(\cm,\cm^\square)$, meaning in particular that any morphism of $\ck$ factorizes as a morphism 
from $\cm$ followed by a morphism from $\cm^\square$. Here, $\cm^\square$ consists of morphisms $g$ having the \textit{right lifting
property} with respect to every morphism $f$ from $\cm$: in any commutative square
$$
\xymatrix@=3pc{
A \ar[r]^{u} \ar[d]_{f}& C \ar[d]^g\\
B\ar[r]_v & D
}
$$
there is a diagonal $d:B \to C$ with $df=u$ and $gd=v$.

\begin{rem}
{
\em
Since the class $\cm$ of pure monomorphisms is left-cancellable, $(\cm,\cm^\square)$ is a weak factorization system if and only if $\ck$ has
enough pure injectives (see \cite[1.6]{AHRT}). Moreover, in an additive locally finitely presentable category $\ck$, $\cm^\square$
is precisely the class of projections $p_1:A\times B\to A$ with $B$ pure injective. 

Indeed, following \cite[1.6]{AHRT}, any such projection belongs to $\cm^\square$. Conversely, $\cm$ contains the class $\cn$ of split monomorphisms, hence $\cm^\square\subseteq\cn^\square$. If $\ck$ is additive, $\cn^\square$ consists of product projections (see \cite[2.7]{RT}), so any morphism from $\cm^\square$ is a projection $p_1:A\times B\to A$. Then the injection $j_2:B\to A\times B$ is the kernel of $p_1$ and, given a pure monomorphism $f:K\to L$ and a morphism $u:K\to B$, the square
$$
\xymatrix@=3pc{
K \ar[r]^{j_2u} \ar[d]_{f}&A\times  B \ar[d]^{p_1}\\
L\ar[r]_0 & A
}
$$
has a diagonal $d:L\to A\times B$. Since $p_1d=0$, $d$ factorizes through its kernel, namely $j_2$, which proves that $B$ is pure injective.

This is analogous to the well-known fact that, if $\cm$ is the class of \emph{all monomorphisms} in an additive locally presentable category, 
then $\cm^\square$ consists of precisely the projections $p_1:A\times B\to A$ where $B$ is \emph{injective}.
}
\end{rem}

Grothendieck abelian categories have enough injectives because they have \textit{effective unions} (see \cite{Ba}).
This means that, given any diagram
\begin{equation} \tag{$*$}
\begin{gathered}
\xymatrix{%@C=3pc@R=3pc{
A \ar[rr]^f  \ar[rd]^{f'} && C \\
 & E \ar [ur]_h&\\
D\ar[uu]^{\bar g} \ar[rr]_{\bar f} && B\ar[ul]_{g'} \ar[uu]_g
}
\end{gathered}
\end{equation}
in which 
\begin{enumerate}\item $f$ and $g$ are regular monomorphisms, 
\item the outer square is a pullback, 
\item and the inner tetragon (consisting of $f',\bar g,g'$ and $\bar f$) is a pushout,\end{enumerate}
the (uniquely defined) induced map $h$ is a regular monomorphism.

Following \cite{BR}, a locally finitely presentable category $\ck$ has \textit{pure effective unions} if whenever $f$ and $g$ in the diagram above
are pure monomorphisms, and the other conditions are satisfied, then $h$ is pure. One can show that, under the assumption of pure effective unions, $\ck$ has enough pure subobjects (\cite[2.4]{BR}). In fact,
the proof shows that pure monomorphisms are cofibrantly generated in this special case.

We note, however, that having pure effective unions is far too strong an assumption:

\begin{exam}\label{unions}
{
\em
$\Ab$ does not have pure effective unions.  Consider the pullback square
$$  
      \xymatrix@=3pc{
        \mathbb{Z} \ar@{}\ar[r]^{} & (\mathbb{Z}\oplus\mathbb{Z}) + \frac{1}{2}(1,1) \\
        0 \ar [u]^{} \ar [r]_{} &
        \mathbb{Z} \ar[u]_{}
      }
      $$
where the upper horizontal morphism and the right vertical morphism are the coproduct injections.  The inclusions on the left and the bottom are clearly pure.       
But $h:\mathbb Z\oplus \mathbb Z\to (\mathbb{Z}\oplus\mathbb{Z}) + \frac{1}{2}(1,1)$, the induced map from the pushout, is not.
}
\end{exam}

This forces us to adopt a more delicate approach.

\section{Representation categories} \label{repr-secn}

A category with finite limits is called \textit{exact} if it has coequalizers of kernel pairs, effective equivalence relations and regular epimorphisms stable under pullbacks, \cite{BGV}. An additive category is abelian if and only if it is exact. An \textit{exact completion} $\ca_{\ex}$ of a finitely complete category $\ca$ is a free exact category over $\ca$. This means that there is a finite
limit preserving functor $E_\ca:\ca\to\ca_{\ex}$ such that for any finite limit preserving functor $F:\ca\to\cb$, $\cb$ an exact category, there is a unique exact (i.e. preserving finite limits and regular epimorphisms) functor $\hat{F}:\ca_{\ex}\to\cb$ with 
$\hat{F}E_\ca\cong F$, \cite{CM}. 

Dualizing, we may speak of \emph{coexactness}, and define the \textit{coexact completion} $C_\ca:\ca\to\ca_{\coex}$ of a finitely cocomplete category $\ca$.

For a small category $\ca$, $\Ind\ca$ is its free completion under directed colimits. This completion is locally finitely presentable
provided that $\ca$ has finite colimits, see \cite{GU}, \cite{MP}, or \cite{AR}.

\begin{defi}\label{repres}
{
\em
Let $\ck$ be a locally finitely presentable category and $\ca$ its (representative) full subcategory of finitely presentable objects.  Let $C_\ca:\ca\to\ca_{\coex}$ be the coexact completion of $\ca$. Then $\cu(\ck)=\Ind\ca_{\coex}$ will be called an \textit{(upper) representation category} of $\ck$ and 
$$
U =\Ind C_\ca:\ck=\Ind\ca\to\Ind(\ca_{\coex})=\cu(\ck)
$$
will be the induced full embedding. 
}
\end{defi}

\begin{exam}
{
\em
 The following example may be illuminating.
 Let $R$ be an associative ring and $\ck=\Mod_R$ the category of right
$R$-modules.
 Then $\ca=\operatorname{mod}_R\subset\ck$ is the full subcategory of
finitely presentable right $R$-modules.
 Consider also the additive category ${}_R\operatorname{mod}$ of finitely
presentable left $R$-modules.

 The upper representation category $\cu(\ck)$ can be described as
the category of ``left $({}_R\operatorname{mod})$-modules,'' namely the
covariant additive functors ${}_R\operatorname{mod}\to\Ab$.
 The functor $U:\ck\to\cu(\ck)$ takes a right $R$-module $M$ to
the tensor product functor $M\otimes_R{-}$.
 The coexact completion $\ca_{\coex}$ of the category $\ca$ is
the category of finitely presentable objects in $\cu(\ck)$; that is,
finitely presentable left $({}_R\operatorname{mod})$-modules.
 The functor $M\otimes_R{-}:{}_R\operatorname{mod}\to\Ab$ is finitely
presentable for any finitely presentable right $R$-module $M$, so
the functor $U$ takes the full subcategory $\ca\subset\ck$ into
the full subcategory $\ca_{\coex}\subset\cu(\ck)$, as it should.
 This describes the functor $E_{\ca}:\ca\to\ca_{\coex}$.

 Clearly, $\cu(\ck)$ is a Grothendieck abelian category.
 The category ${}_R\operatorname{mod}$ has weak cokernels (in fact, it has
cokernels); therefore, the category $\ca_{\coex}$ is abelian,
too~\cite[Lemma~2.2]{K}.
}
\end{exam}

This construction was introduced in the unpublished paper \cite{H}. In the abelian case, it coincides with the upper representation categories of, e.g. \cite{C} and \cite{K}.

\begin{lemma}\label{lfp}
The category $\cu(\ck)$ is locally finitely presentable and $U:\ck\to\cu(\ck)$ preserves colimits.
\end{lemma}
\begin{proof}
Since $\ca_{\coex}$ is finitely cocomplete, $\cu(\ck)$ is locally finitely presentable. Since $C:\ca\to\ca_{coex}$ preserves finite colimits, $U$ does the same and, since it preserves directed colimits as well, it preserves arbitrary colimits. 
\end{proof}

\begin{rem}\label{coexact}
{
\em
Moreover, $\cu(\ck)$ is a free locally finitely presentable coexact category over $\ck$. In fact, $\cu(\ck)$ is coexact and
$U$ preserves colimits. Let $H:\ck\to\Ind\cb$ be a colimit preserving functor to a locally finitely presentable coexact category.
Since $\ca$ has finite colimits preserved by the inclusion $Y:\ca\to\Ind\ca$ (see the beginning of the proof of \cite{AR} 4.13),
the composition $HY$ preserves finite colimits. Thus there is an essentially unique coexact functor $\bar{H}:\ca_{coex}\to\Ind\cb$. Hence there is an essentially unique coexact functor $G:\cu(\ca)\to\Ind\cb$ such that $GY\cong H$.
}
\end{rem}
An object of a category is called \textit{regularly fp-injective} if it is injective with respect to regular monomorphisms with finitely presentable domains and codomains.

\begin{propo}\label{reginj}
$U(\ck)$ consists of precisely the regularly fp-injectives in $\cu(\ck)$.
\end{propo}
\begin{proof}
The category $\ca_{\coex}$ is a full subcategory of the category $(\Set^\ca)^{\op}$ and $C_\ca$ is the codomain restriction
of the Yoneda embedding $\ca\to (\Set^\ca)^{\op}$ (see \cite[17.7]{ARV}. Moreover, the embedding $\ca_{\coex}\to(\Set ^\ca)^{\op}$ is coexact, i.e., preserves finite colimits and regular monomorphisms. Since hom-functors are regularly injective in $(\Set^\ca)^{\op}$, objects $C(A)$ are regularly injective in $\ca_{\coex}$. Thus they are regularly fp-injective in $\cu(\ca)$. Since regularly fp-injectives are closed under directed colimits, objects from $U(\ck)$ are regularly
fp-injective in $\cu(\ca)$.

Conversely, consider a regularly fp-injective object $K$ in $\cu(\ck)$. and a morphism $f:X\to K$ where $X$ is in $\ca_{\coex}$. There is
a regular monomorphism $h:X\to C(A)$ where $A$ is in $\ca$ (again, this is dual to the property of an exact completion, \cite{ARV} 17.8). Since $K$ is regularly fp-injective, there is $g:U A\to K$ with $gh=f$. Hence $U(\ca)\downarrow K$ is cofinal 
in $\ca_{\coex}\downarrow K$ and therefore $K$ belongs to $U(\ck)$.
\end{proof} 

The result above was proven in the abelian case in \cite{C}.

\begin{rem}\label{wrefl}
{
\em
Following \ref{reginj}, $U(\ck)$ is weakly reflective in $\cu(\ck)$ with a weak reflector $R:\cu(\ck)\to U(\ck)$ preserving directed colimits.
}
\end{rem}
 
The following result is in \cite{H}; the abelian case appears in \cite{C}.
 
\begin{lemma}\label{pure}
A morphism $h$ is a pure monomorphism in $\ck$ if and only if $Uh$ is a regular monomorphism in $\cu(\ck)$.
\end{lemma}
\begin{proof}
If $h$ is a pure monomorphism then, following \cite{AR} 2.37, $Uh$ is a pure monomorphism, hence a regular monomorphism (see Remark~\ref{purermk}(1)). Conversely, let $Uh$ be a regular monomorphism and consider a commutative square
$$  
      \xymatrix@=3pc{
        K \ar@{}\ar[r]^h{} & L \\
        A \ar [u]^{u} \ar [r]_{f} &
        B \ar[u]_{v}
      }
      $$
in $\ck$ where $A$ and $B$ are in $\ca$. Let 
$$
UA \xrightarrow{\ f_1\ } X\xrightarrow{\ f_2\ } UB
$$
be an (epimorphism, regular monomorphism)-factorization of $Uf$ in $\ca_{\coex}$.      
Since $Uh$ is a regular monomorphism, there is $t:X\to UK$ such that $tf_1=Uu$. Since $UK$ is regularly fp-injective, there is 
$w:UB\to UK$ with $wf_2=t$. Then $w=U\bar{w}$ and $\bar{w}f=u$. Thus $h$ is a pure monomorphism.
\end{proof}

\begin{rem}\label{regular}
{
\em
The category $\cu(\ck)$ is coexact, so it has an (epimorphism, regular monomorphism)-factorization.  In particular, then, the class of regular monomorphisms in $\cu(\ck)$ is left cancellative. Hence regular monomorphisms form the left part of a weak factorization system
in $\cu(\ck)$ if and only if $\cu(\ck)$ has enough regular injectives (see \cite{AHRT} 1.6). As mentioned in the introduction, the latter means that any object of $\cu(\ck)$ is a regular subobject of a regular injective object.
}
\end{rem}

Since the class of pure monomorphisms in $\ck$ is left cancellative, it forms a left part of a weak factorization system if and only if $\ck$ has enough pure injectives.

\begin{theo}\label{pureinj}
Assume that $\cu(\ck)$ has enough regular injectives. Then $K$ is pure injective in $\ck$ if and only if $UK$ is regular injective
in $\cu(\ck)$.

Moreover, $\ck$ has enough pure injectives.
\end{theo}
\begin{proof}
Let $UK$ be regular injective. Consider a pure monomorphism $h:M\to N$ in $\ck$ and a morphism $f:M\to K$. Since $Uh$ is a regular monomorphism (see \ref{pure}), there is $g:UN\to UK$ with $gU(h)=Uf$. For $\bar{g}:N\to K$, we have $\bar{g}h=f$. Thus $K$ is pure
injective in $\ck$.

Conversely, let $K$ be pure injective. There is a regular monomorphism $h:UK\to X$ where $X$ is regularly injective in $\cu(\ck)$,
Following \ref{reginj}, $X=UL$ for some $L$ in $\ck$. Moreover, $h=U\bar{h}$ for $\bar{h}:K\to L$ and $\bar{h}$ is a pure monomorphism
(see \ref{pure}). Since $K$ is pure injective, $\bar{h}$ is a split monomorphism. Hence $h$ is a split monomorphism and $UK$ must be regular injective.

Finally, given $K$ in $\ck$, there is a regular monomorphism $h:UK\to UL$ where $UL$ is regular injective. Thus $\bar{h}:K\to L$ is
a pure monomorphism and $L$ is pure injective. It follows that $\ck$ has enough pure injectives.
\end{proof}

In the abelian case, $\cu(\ck)$ has enough regular injectives, see \cite{C}.

\begin{rem}\label{pureinj1}
{
\em
(1) The category of oriented multigraphs is locally finitely presentable and does not have enough pure injectives (see \cite {BR}
2.7(1)).

(2) If $\ck$ is (pre)additive, then $\ca_{\coex}$ is abelian (see e.g. \cite{RV}) and thus $\cu(\ck)$ is abelian as well.
Thus $\cu(\ck)$ has enough regular injectives, from which it follows that $\ck$ has enough pure injectives.
}
\end{rem}

\begin{theo}\label{stable}
Pure monomorphisms in $\ck$ are cofibrantly generated provided that regular monomorphisms in $\cu(\ck)$ are cofibrantly generated.
\end{theo}
\begin{proof}
Assume that regular monomorphisms are cofibrantly generated in the (locally finitely presentable) category $\cu(\ck)$.  Since the embedding $U:\ck\to\cu(\ck)$ preserves colimits, \cite{MR} 3.8 implies that we can pull this cofibrantly generated class back along $U$ to obtain one on $\ck$: the class of all maps $h$ in $\ck$ with $Uh$ a regular monomorphism in $\cu(\ck)$ will be cofibrantly generated.  In light of Lemma~\ref{pure}, this says precisely that the class of pure monomorphisms in $\ck$ is cofibrantly generated.
\end{proof}

\begin{rem}
{
\em
We note that the converse of Theorem~\ref{stable} holds as well: regular monomorphisms in $\cu(\ck)$ are cofibrantly generated provided that pure monomorphisms in $\ck$ are. A proof will appear in forthcoming work.
}
\end{rem}

\begin{coro}\label{abelian}
Any locally finitely presentable additive category has cofibrantly generated pure monomorphisms.
\end{coro}
\begin{proof}
Since $\cu(\ck)$ is abelian (by Remark~\ref{pureinj1}(2)), it has effective unions. Thus regular monomorphisms are cofibrantly generated in $\cu(\ck)$ and we can apply \ref{stable}.
\end{proof}

\begin{coro} \label{add-pureinj}
Any locally finitely presentable additive category has enough pure
injectives.
\end{coro}

\begin{proof}
 Follows from Theorem~\ref{pureinj} or from Corollary~\ref{abelian}.
\end{proof}

Given a locally finitely presentable additive category $\ck$, let $\cx$ be a cofibrantly generating set of pure monomorphisms.  One can easily verify that an object in $\ck$ is pure injective if and only if it is injective with respect to each morphism in $\cx$.  That is,

\begin{coro}\label{corsinj}
In any locally finitely presentable additive category, pure injectives form a small-injectivity class.
\end{coro}
 
\begin{rem}\label{remtestset}
{
\em
This was recently proven for module categories (using slightly different language, i.e. \emph{test sets}) in \cite{ST}.
}
\end{rem}

It follows directly from Corollary~\ref{corsinj} that pure injectives in any locally finitely presentable additive category form an accessible category (see
\cite[4.16]{AR}).

\section{Examples of locally finitely presentable additive categories}

 In this section, we exhibit two examples of ``badly behaved'' locally
finitely presentable additive categories that illustrate how badly the situation above breaks down if we consider not the pure monomorphisms and injectives but rather, say, the regular ones.
 Both examples are constructed using the following general technique.
 Let $Q$ be a finite quiver without relations (i.e., simply a finite
oriented graph) and $F$~be a field.
 Assume that $Q$ is acyclic and has finite representation type.
 This means that the path algebra $F[Q]$ is finite-dimensional and
all $F[Q]$-modules are direct sums of finite-dimensional ones.
 Then the category $\cl$ of $F[Q]$-modules is pure semisimple, that is,
all pure exact sequences of $F[Q]$-modules are split and all $F[Q]$-modules
are pure injective.

 We pick a subset $\ci$ of isomorphism classes of indecomposable
$F[Q]$-modules and let $\ck$ be the full subcategory of $\cl$ whose
objects are all the direct sums of modules from~$\ci$.
 Since all pure exact sequences in $\cl$ are split and $\ck$ is closed under
direct sums and direct summands in $\cl$, it follows that the full
subcategory $\ck\subset\cl$ is closed under directed colimits.
 Consequently, $\ck$ is a finitely accessible additive category,
and finitely presentable objects in $\ck$ are precisely
the finite-dimensional modules (i.e., the finite direct sums of modules
from~$\ci$).
 Moreover, we will pick our examples in such a way that $\ck$ is either
reflective or coreflective in $\cl$; then $\ck$ is locally finitely
presentable.
 
 All pure monomorphisms are split in $\ck$, and all objects are pure
injective. But nonpure monomorphisms and injective objects may behave badly (depending on how the subset $\ci$ is chosen). Note also that, while split monomorphisms are cofibrantly generated here, they need not be in general (consider for example the category of abelian groups \cite[2.6]{Ro}).

 The following general definitions related to additive categories will be
useful for our discussion.
 An additive category is said to be \emph{preabelian} if all of its morphisms have kernels and cokernels (equivalently, there are finite limits
and colimits).
 Any locally presentable additive category has all limits and colimits, and is therefore preabelian.

 Let $z:X\to Y$ be a morphism in a preabelian category $\ck$.
 Let $k:K\to X$ and $c:Y\to C$ be the kernel and cokernel of $z$.
 Let $X\to J=\operatorname{coim}(z)$ be the cokernel of $k$ and
$\operatorname{im}(z)=I\to Y$ be the kernel of $c$; so the morphism $f$
decomposes as $X\to J\to I\to Y$.
 The category $\ck$ is said to be \emph{left semi-abelian} if
the morphism $J\to I$ is a monomorphism for all morphisms $z$.
 Dually, $\ck$ is \emph{right semi-abelian} if $J\to I$ is an epimorphism
for all morphisms $f$.
 A discussion of these notions can be found in the papers~\cite{Ru,KW} and
also in~\cite[Example~A.5(7)]{Po}.

 In the language of category theory, an additive category being left
semi-abelian means that it admits (regular epimorphism,
monomorphism)-factorizations, while right semi-abelianity means
that it admits (epimorphism, regular monomorphism)-factor\-ization.
 For any left semi-abelian category $\ck$ and a diagram~($*$) as in
Section~\ref{repr-secn} such that conditions~(2) and~(3) are satisfied,
the morphism $h$~is a monomorphism, as one can see by applying the left
semi-abelianity condition to the morphism $z=(f,g):A\oplus B\to C$.
 But $h$~does not need to be a regular monomorphism.

 Conversely, any preabelian additive category $\ck$ with effective unions
is left semi-abelian.
 Indeed, let $z:X\to Y$ be a morphism in $\ck$.
 Then, in order to show that the induced morphism $\operatorname{coim}(z)
\to\operatorname{im}(z)$ is a monomorphism, one can apply the condition
of existence of effective unions to the pair of spit monomorphisms
$f=(1,0):X\to X\oplus Y$ and $g=(1,z):X\to X\oplus Y$ between the objects
$A=B=X$ and $C=X\oplus Y$ (cf.\ Example~\ref{right-exam} below).

\begin{exam}
{\em
 This example was suggested in \cite[Example~A.5(7)]{Po}.
 Let $Q$ be the quiver $\bullet\to\bullet\to\bullet$.
 This means that $F[Q]$-modules $V$ are sequences of $F$-vector spaces and
linear maps $V^{(1)}\to V^{(2)}\to V^{(3)}$.
 There are $6$ indecomposable $F[Q]$-modules, which can be denoted by
$E_1$, $E_2$, $E_3$, $E_{12}$, $E_{23}$, and $E_{123}$.
 Here the subindex $J$ denotes a subset of vertices of the quiver $Q$,
and the representation $E_J$ is described by the rules $E_J^{(i)}=0$ for
$i\notin J$ and $E_J^{(i)}=F$ for $i\in J$.
 The map $E_J^{(i)}\to E_J^{(i+1)}$, \ $1\le i\le 2$, is the identity map
whenever $E_J^{(i)}=F=E_J^{(i+1)}$.

 We let the set $\ci$ consist of all the indecomposable $F[Q]$-modules
except~$E_{12}$.
 Notice that $E_{12}$ is an injective object of the abelian category $\cl$;
therefore, the full subcategory $\ck\subset\cl$ consisting of direct sums of modules in $\ci$ is closed under
subobjects (hence under all limits).
 Moreover, $\ck$ is reflective in $\cl$; the reflector $R:\cl
\to\ck$ takes the object $E_{12}$ to $E_1$.
 For any morphism $z:X\to Y$ in $\ck$, the object
$\operatorname{cok}_\ck(z)=R(\operatorname{cok}_\cl(z))$ is a certain
quotient module of $\operatorname{cok}_\cl(z)$, while
$\ker_\ck(z)=\ker_\cl(z)$.
 Hence $\operatorname{im}_\cl(z)$ is a submodule of
$\operatorname{im}_\ck(z)$, while $\operatorname{coim}_\ck(z)=
\operatorname{coim}_\cl(z)$.
 Therefore, the category $\ck$ is left, but not right, semi-abelian.
 In fact, taking $z$ to be a nonzero morphism $E_3\to E_{123}$,
the related morphism $J=E_3\to E_{23}=I$ is a regular monomorphism
and not an epimorphism in $\ck$.

 In particular, the morphism~$h$ in the diagram~($*$) in the category $\ck$
is always a monomorphism.
 But it is not always a regular monomorphism, even if the morphisms
$f$ and $g$ are; so $\ck$ does not have effective unions.
 Indeed, let $C$ be the quiver representation $F\xrightarrow{1}F
\xrightarrow{(1,1)}F\oplus F$.
 Let $A$ and $B$ be the subrepresentations $0\to 0\to F\oplus 0$
and $0\to 0\to 0\oplus F$ of $C$.
 We have $C\simeq E_{123}\oplus E_3$ and $A\simeq B\simeq E_3$; so
$A$, $B$, and $C$ are objects of the full subcategory $\ck\subset\cl$.
 Both $A$ and $B$ are split subobjects of $C$; so the morphisms $f$
and $g$ in the diagram~($*$) are split monomorphisms.
 Furthermore, the intersection $D=A\cap B$ vanishes; hence the object
$E$ is $E=A\oplus B$.
 The monomorphism $h:E\to C$ is isomorphic to the direct sum of the identity
morphism $E_3\to E_3$ and a nonzero morphism $E_3\to E_{123}$.
 The morphism $E_3\to E_{123}$ is a nonregular monomorphism in $\ck$;
hence $h$~is not a regular monomorphism, either.

 Furthermore, consider a nonzero morphism $k:E_{23}\to E_{123}$ and
a nonzero morphism $p:E_{23}\to E_2$.
 Then $k$~is a regular monomorphism and $p$~is a regular epimorphism
in~$\ck$.
 The pushout of $k$ by $p$ in the category $\cl$ is a nonzero morphism
$l':E_2\to E_{12}$.
 But $E_{12}$ is not an object of $\ck$.
 Applying the reflector $R$, we see that the pushout of $k$ by $p$
in $\ck$ is the zero morphism $k':E_2\to E_1$.
 So neither monomorphisms nor regular monomorphisms are stable under
pushouts in $\ck$.

 Finally, we claim that the object $E_2$ is not a subobject of any regular
injective object in $\ck$.
 In fact, any morphism in $\ck$ having $E_2$ as the domain is either zero,
or a split monomorphism.
 It remains to show that $E_2$ is not a regular injective in $\ck$.
 Indeed, $k:E_{23}\to E_{123}$ is a regular monomorphism and 
the morphism $p:E_{23}\to E_2$ cannot be extended over $k$.
 One can recall that in an abelian category an object is injective if and
only if it always splits when embedded as a subobject into any other object.
 The object $E_2\in\ck$ satisfies the latter condition, but it is not
regular injective.
 So injectivity of objects in $\ck$ does not behave as one might expect.}
\end{exam}
 
\begin{exam} \label{right-exam}
{\em
 Let $Q$ be the same quiver as in the previous example.
 We let the set $\ci$ consist of all the indecomposable $F[Q]$-modules
except $E_{23}$.

 Notice that $E_{23}$ is a projective object of the abelian category
$\cl$; therefore, the full subcategory $\ck\subset\cl$ is closed under
quotient objects (hence under all colimits).
 The full subcategory $\ck$ is coreflective in $\cl$; the coreflector
$S:\cl\to\ck$ takes the object $E_{23}$ to $E_3$.
 For any morphism $z:X\to Y$ in $\ck$, the object $\ker_\ck(z)=
S(\ker_\cl(z))$ is a certain submodule of $\ker_\cl(z)$, while
$\operatorname{cok}_\ck(z)=\operatorname{cok}_\cl(z)$.
 Hence $\operatorname{coim}_\cl(z)$ is a quotient module of
$\operatorname{coim}_\ck(z)$, while $\operatorname{im}_\ck(z)=
\operatorname{im}_\cl(z)$.
 Therefore, the category $\ck$ is right, but not left, semi-abelian.
 In fact, taking $z$ to be a nonzero morphism $E_{123}\to E_1$,
the related morphism $\operatorname{coim}_\ck(z)=E_{12}\to E_1=
\operatorname{im}_\ck(z)$ is a regular epimorphism and not
a monomorphism in $\ck$.

 All monomorphisms in $\ck$ are regular, and monomorphisms are stable
under push\-outs.
 There are enough (regular) injectives in $\ck$, which coincide with
the injective objects of $\cl$.
 However, the category $\ck$ does not have effective unions.
 Indeed, let $C$ be the quiver representation $F\oplus F\xrightarrow{(1,1)}
F\xrightarrow 1 F$.
 Let $A$ and $B$ be the submodules $F\oplus 0\to F\to F$ and
$0\oplus F\to F\to F$ of $C$.
 We have $C\simeq E_1\oplus E_{123}$ and $A\simeq B\simeq E_{123}$;
so $A$, $B$, and $C$ are objects of the full subcategory $\ck\subset\cl$.
 Both $A$ and $B$ are split subobjects of $C$; so the morphisms $f$
and $g$ in the diagram~($*$) are split monomorphisms.
 Furthermore, the intersection of the submodules $A$ and $B$ in $C$
is isomorphic to $E_{23}$.
 But $E_{23}$ is not an object of $\ck$.
 Applying the coreflector $S$, we see that the object $D$ in
the diagram~($*$) in the category $\ck$ is $D\simeq E_3$.
 Hence one easily computes that the object $E$ in~($*$) is
isomorphic to $E_{12}\oplus E_{123}$.
 The morphism $h:E\to C$ is not a monomorphism; in fact, it is
a regular epimorphism in~$\ck$.
}
\end{exam}

\bigskip

%\bibliographystyle{alpha}
%\bibliography{purecofgen.bib}

\end{document}